\documentclass[11pt,reqno]{amsart}
\usepackage{amssymb}
{\theoremstyle{plain}
 \newtheorem{theorem}{Theorem}
 \newtheorem{corollary}{Corollary}
 
\newtheorem{lemma}{Lemma}
}

\DeclareMathOperator{\re}{Re}

\begin{document}

\title{The number of prime factors of integers with dense divisors}

\begin{abstract}
We show that for integers $n$, whose ratios of consecutive divisors are bounded above by an arbitrary constant,
the normal order of the number of prime factors is $C \log \log n$,
where $C=(1-e^{-\gamma})^{-1} = 2.280...$ and $\gamma$ is Euler's constant. 
We explore several applications and resolve a conjecture of Margenstern about 
practical numbers.
\end{abstract}

\author{Andreas Weingartner}
\address{ 
Department of Mathematics,
351 West University Boulevard,
 Southern Utah University,
Cedar City, Utah 84720, USA}
\email{weingartner@suu.edu}
\date{November 11, 2021}
\maketitle

\section{Introduction}

We say that a positive integer $n$ is $t$-dense if the ratios of consecutive divisors of $n$ do not exceed $t$. 
Let $\mathcal{D}(x,t)$ denote the set of $t$-dense integers $n\le x$ and write $D(x,t)= |\mathcal{D}(x,t)|$.
Let $\omega(n)$ (resp. $\Omega(n)$) be the number of prime factors of $n$, counted without (resp. with) multiplicity. 

Theorem \ref{thmave} gives the average and normal order of $\omega(n)$ and $\Omega(n)$ for the $t$-dense integers.
We write $\log_2 x$ for $\log\log x$ and define $E(x,t)$, the approximate expected value of $\omega(n)$ on 
 $\mathcal{D}(x,t)$, by  
$$
E(x,t):=C\log_2 x- (C-1)\log_2 t, \quad C:=(1-e^{-\gamma})^{-1}=2.280291...
$$
\begin{theorem}\label{thmave}
Let $\xi(x) \to \infty$. Uniformly for $x\ge t\ge 2$,
\begin{equation}\label{eqave}
\frac{\sum_{n\in \mathcal{D}(x,t)}\omega(n)}{D(x,t)}
=E(x,t)+O\left(1\right)
\end{equation}
and
\begin{equation}\label{eqnorm}
\Bigl|\left\{ n \in \mathcal{D}(x,t): |\omega(n)-E(x,t)|> \xi(x)\sqrt{\log_2 x}\right\}\Bigr| \ll \frac{D(x,t)}{\xi(x)^2}.
\end{equation}
These results also hold with $\Omega$ in place of $\omega$. 
\end{theorem}

Note that $\log_2 x \le E(x,t) \le C \log_2 x +O(1)$, as $x\ge t\ge 2$.
If $t=x$, then $E(x,t)=\log_2 x$ and $ \mathcal{D}(x,t)=[1,x]\cap \mathbb{N},$ so
Theorem \ref{thmave}  contains the well-known results about the average and normal order of $\omega(n)$ on $\mathbb{N}$. 
If $t\ge 2$ is constant, then $E(x,t)= C\log_2 x +O(1)$, so that the average and normal order of $\omega(n)$ 
for $t$-dense integers $n$ is $C \log_2 n$. 

The $t$-dense integers are a special case of a family of integer sequences that arise as follows. 
Let $\theta$ be an arithmetic function.
Let $\mathcal{B}=\mathcal{B}_\theta$ be the set of positive integers containing $n=1$ and all those $n \ge 2$ with prime factorization $n=p_1^{\alpha_1} \cdots p_k^{\alpha_k}$, $p_1<p_2<\ldots < p_k$, which satisfy 
\begin{equation}\label{Bdef}
p_{i} \le \theta\big( p_1^{\alpha_1} \cdots p_{i-1}^{\alpha_{i-1}} \big) \qquad (1\le i \le k).
\end{equation}
We write $\mathcal{B}(x)=\mathcal{B}\cap [1,x]$ and $B(x)=|\mathcal{B}(x)|$. 
When $\theta(n)=nt$, then $\mathcal{B}$ is the set of $t$-dense integers \cite{Saias1,Ten86,PDD}. 
If $\theta(n)=\sigma(n)+1$, where $\sigma(n)$ is the sum of the positive divisors of $n$, then 
$\mathcal{B}$ is the set of practical numbers \cite{Mar, Saias1, Sier, Sri,  Stew, Ten86,PDD}, i.e. integers $n$ such that every $m\le n$ can be expressed as a sum of distinct positive divisors of $n$.

We will derive Corollaries \ref{corpra} through \ref{cortau} from Theorem \ref{thmave} in Section \ref{seccorpra}.
\begin{corollary}\label{corpra}
 Assume $\theta$ satisfies $\max(2,n)\le \theta(n) \le n f(n)$, where $f$ is non-decreasing.  Let $\xi(x) \to \infty$. We have 
 \begin{equation}\label{cor1ave}
\frac{\sum_{n\in \mathcal{B}(x)}\omega(n)}{B(x)} =
  C \log_2 x \left\{ 1+ O\left(\left(\frac{\log f(x)}{\log_2 x}\right)^{1/3}\right)\right\}
 \end{equation}
 and
\begin{equation}\label{cor1norm}
\Bigl|\bigl\{ n \in \mathcal{B}(x): |\omega(n)-C\log_2 x|>  \xi(x)\sqrt{\log_2 x} \bigr\}\Bigr| 
\ll  B(x)\frac{\log f(x)}{\xi(x)^2}.
\end{equation}
These results also hold with $\Omega$ in place of $\omega$. 
\end{corollary}
Conjecture 5 of Margenstern \cite{Mar} proposes that, for practical numbers, 
$\sum_{n\in \mathcal{B}(x)}\omega(n) \sim \mu x/(\log x)^\eta$,
for some constants $\mu >0$ and $1/2 < \eta < 1$. 
The estimate \eqref{cor1ave} disproves this conjecture, since 
$B(x)\sim c_\theta x/\log x$ by \cite[Thm. 1.2]{PDD}, if $\max(2,n)\le  \theta(n) \ll n (\log 2n)/(\log_2 3n)^{1+\varepsilon}$
for $n\ge 1$. 

Corollary \ref{cor2} shows that almost all large practical numbers $n$ have about $C\log_2 n$ prime factors. 

\begin{corollary}\label{cor2}
If $\theta$ satisfies $\max(2,n)\le \theta(n) \ll n  (\log 2n)^{o(1)}$, 
then the average and normal order of $\omega(n)$ on $\mathcal{B}$ is $C\log_2 n$.
That is, as $x\to\infty$,
 \begin{equation}\label{cor2ave}
\sum_{n\in \mathcal{B}(x)}\omega(n) \sim \sum_{n\in \mathcal{B}(x)} C \log_2 n
  \end{equation}
 and all but $o(B(x))$ integers $n \in \mathcal{B}(x)$ satisfy
\begin{equation}\label{cor2norm}
\omega(n)=(1+o(1)) C \log_2 n .
 \end{equation}
These results also hold with $\Omega$ in place of $\omega$. 
\end{corollary}

Let $\tau(n)$ be the number of positive divisors of $n$.
\begin{corollary}\label{cortau}
If $\theta$ satisfies $\max(2,n)\le \theta(n) \ll n  (\log 2n)^{o(1)}$, then
\begin{equation}\label{cortauave}
 (\log x)^{C \log 2 -o(1)} < \frac{\sum_{n \in \mathcal{B}(x)} \tau(n)}{B(x)} \ll (\log x)^{e \log 2},
 \end{equation}
as $x \to \infty$, and all but $o(B(x))$ integers $n\in \mathcal{B}(x)$ satisfy
\begin{equation}\label{cortaunorm}
\tau(n) = (\log n)^{C \log 2 + o(1)} = (\log n)^{1.580577... + o(1)}.
\end{equation}
\end{corollary}

Conjecture 4 of Margenstern \cite{Mar} says that, in the case of practical numbers, 
$\sum_{n \in \mathcal{B}(x)} \tau(n) \sim \nu x (\log x)^{\delta}$, for constants $1/2 < \nu, \, \delta < 1$. 
The estimate \eqref{cortauave} implies $\delta \in [C \log 2 -1, e\log 2 -1] = [0.580...,0.884...]$, since $B(x) \asymp x/\log x$.
In \cite{avetaupaper}, we prove this conjecture with $\delta=0.713...$ and some constant $\nu>0$. 

The next two corollaries are improvements to the lower bounds of Theorems 1 and 3 of \cite{PW}. 
The proofs of both of these theorems rely on the fact that almost all $n \in \mathcal{B}$ satisfy
$\Omega(n)< (e + \varepsilon) \log_2 n$. With Corollary \ref{cor2}, this can be improved
to $\Omega(n)< (C + \varepsilon) \log_2 n$, 
under the assumption $\theta(n) \ll n (\log n)^{o(1)}$.
In the lower bound of  \cite[Thm. 1]{PW} for the count of practical numbers that are also shifted primes,
which has an exponent of $(e+1)\log(e+1)-e\log(e)+1+\varepsilon = 3.16470...$,
we can replace $e$ by $C$ to get $3.01711...$. 

\begin{corollary}\label{cor1PW}
Fix a nonzero integer $h$ and assume $\theta$ satisfies 
\begin{equation}\label{thetacondcor}
\max(2,n) \le \theta(n) \ll n (\log 2n)^{o(1)}, \quad  \theta(mn)\ll m^{O(1)} \theta(n)\quad (n,m\in \mathbb{N}).
\end{equation}
We have
\begin{equation*}
\frac{x}{(\log x)^{3.01712}} \ll_h  
\bigl|\{p \le x: p \mbox{ prime},\  p-h \in \mathcal{B} \}\bigr| \ll_h \frac{x}{(\log x)^2},
\end{equation*}
where $h$ is not divisible by $\prod_{p\le \theta(1)} p$ in the lower bound.  
\end{corollary}

Similarly, in the lower bound of \cite[Thm. 3]{PW} for the count of twin practical numbers, 
which has an exponent of $2+4 e \log2 + \varepsilon = 9.53667...$, 
we can replace $e$ by $C$ to get $8.32230...$

\begin{corollary}\label{cor2PW}
Fix a nonzero integer $h$ and assume $\theta$ satisfies \eqref{thetacondcor}.
We have
\begin{equation*}
\frac{x}{(\log x)^{8.32231}}\ll_h 
\bigl|\{n\le x: n \in \mathcal{B}, n+h \in \mathcal{B}\}\bigr| 
\ll_h \frac{x}{(\log x)^2}.
\end{equation*}
For the lower bound, assume that (i) $n\in \mathcal{B}$ and $m\le 3n/|h|$ imply $mn\in \mathcal{B}$, and
(ii) if $\theta(1)<3$, then $h \in 2\mathbb{Z}$ if $\theta(2)\ge 3$, and  $h \in 4\mathbb{Z}$ if $\theta(2)< 3$.
\end{corollary}
Conditions (i) and (ii) in Corollary \ref{cor2PW} are satisfied by the practical numbers and by the $2$-dense integers 
for any nonzero even integer $h$, and by the $t$-dense integers for any nonzero integer $h$, provided $t \ge 3$. 

The $\varphi$-practical numbers \cite{PTW, Thompson} are integers $n$ such that $x^n-1$ has a divisor 
in $\mathbb{Z}[x]$ of every degree up to $n$. Although not an example of a set $\mathcal{B}_\theta$,  
they are a superset of $\mathcal{B}_{\theta_1}$ with $\theta_1(n)=n+1$, and a subset
of $\mathcal{B}_{\theta_2}$ with $\theta_2(n)=n+2$. 
Therefore, Corollaries \ref{corpra} through \ref{cor2PW} also 
apply to the $\varphi$-practical numbers, provided  $h$ is odd  in the lower bound of Corollary \ref{cor1PW},
while  $h$ is even in the lower bound of Corollary \ref{cor2PW}.

Theorem \ref{thmave} is a consequence of Theorems \ref{thmDqdv} and \ref{thmcpq}. 
Theorem  \ref{thmDqdv} gives an estimate for 
$$D_q(x)=D_q(x,t):=|\{n\in \mathcal{D}(x,t): q|n \}|,$$
when $q$ has a bounded number of prime factors. 
As in the case $q=1$ (see \cite[Thm. 1.3]{PDD}), the main term contains the function $d(v)$, which 
is defined by $d(v)=0$ for $v<0$ and 
\begin{equation}\label{dinteq}
d(v)= 1-\int_0^{\frac{v-1}{2}} \frac{d(u)}{u+1} \  w\left(\frac{v-u}{u+1}\right)  du \qquad (v\ge  0),
\end{equation}
where $w(u)$ denotes Buchstab's function. 

\begin{theorem}\label{thmDqdv}
Let $k\in \mathbb{N}\cup \{0\}$ be fixed. 
Uniformly for $x \ge 1$, $t\ge 2$, $q\in \mathbb{N}$ with $\Omega(q)=k$, $v=\log x /\log t$, we have
\begin{align}
 D_q(x,t) 
  &= x d(v) \eta_{q,t} \left\{1 + O_k\left( \frac{1}{\log xt} + \frac{\log 2q \log qt}{\log^2 xt}\right) \right\}+O(1), \label{eqDqdv0} \\ 
 D_q(x,t) & = x d(v) \eta_{q,t} \left\{1 + O_k\left( \frac{\log 2q}{\log 2x} \right) \right\},\label{eqDqdv}
\end{align}
where 
\begin{equation}\label{etaeq}
q^{-1}\ll \eta_{q,t} \ll_k q^{-1}.
\end{equation}
\end{theorem}

\begin{corollary}\label{corDqdv}
Let $k\in \mathbb{N}\cup \{0\}$ be fixed. 
Uniformly for $x \ge 1$, $t\ge 2$, $q\in \mathbb{N}$ with $\Omega(q)=k$, we have
\begin{equation}\label{eqcorDqdv}
 D_q(x,t) = \frac{c_{q} x}{\log xt} \left\{1 + O_k\left( \frac{1}{\log xt} + \frac{ \log^2 qt}{\log^2 xt}\right) \right\} +O(1),
\end{equation}
where 
\begin{equation}\label{cqeq}
c_q=c_{q,t} = C \eta_{q,t} \log t ,\quad  q^{-1}\log t \ll c_{q} \ll_k q^{-1}\log t, 
\end{equation}
and 
\begin{equation}\label{eqcorDqdv2}
 D_q(x,t) \ll_k \frac{ x \log t}{q\log xt}.
\end{equation}
\end{corollary}

The estimates \eqref{eqcorDqdv} and \eqref{cqeq} follow from \eqref{eqDqdv0} and \eqref{etaeq},
since
$
d(v) =C(v+1)^{-1}\left\{1+O\left((v+1)^{-2}\right)\right\}
$
by \cite[Thm. 1]{IDD3}. 
The upper bound \eqref{eqcorDqdv2} follows from \eqref{eqDqdv}, because $D_q(x,t)=0$ if $q>x$. 

Theorem \ref{thmcpq} gives estimates for $c_q$ when $q$ is a prime or a product of two primes. 
These estimates are needed to derive Theorem \ref{thmave} from Theorem \ref{thmDqdv}.

\begin{theorem}\label{thmcpq}
Let  $p\le q$ be primes. The constant factor in \eqref{eqcorDqdv} satisfies
\begin{equation}\label{c1tasymp}
c_\theta := c_1 = C (\log t -\gamma) +  O\left(e^{-\sqrt{\log t}}\right),
\end{equation}
\begin{equation}\label{cqtasymp1}
q c_q = C c_\theta
\left\{1+O\left(\frac{1}{\log q}+\frac{\log^2 t}{\log^2 q}\right)\right\}
\end{equation}
\begin{equation}\label{cqtasymp2}
q c_q = c_\theta + C\log q  + O\left(\exp\left(-\sqrt{\log t}\right)\right)\quad (q\le t),
\end{equation}
\begin{equation}\label{cpqtasymp1}
pq c_{pq} = C^2 c_\theta \left\{1+O\left(\frac{1}{\log p}+\frac{\log^2 t}{\log^2 p}
+ \frac{\log^2 p}{\log^2 q}\right)\right\},
\end{equation}
\begin{equation}\label{cpqtasymp2}
pq c_{pq} = \left(C c_\theta + C^2 \log p\right)
\left\{1+O\left(\frac{1}{\log q}+\frac{\log^2 t}{\log^2 q}
+e^{-\sqrt{\log t}}\right)\right\}
\quad (p\le t),
\end{equation}
\begin{equation}\label{cpqtasymp3}
pq c_{pq} =  c_\theta + C \log pq + O\left(e^{-\sqrt{\log t}}\right)\quad (p\le q\le t).
\end{equation}
\end{theorem}

In Section \ref{seccorpra} we derive Corollaries \ref{corpra}, \ref{cor2} and \ref{cortau} from Theorem \ref{thmave}.
Section \ref{secmult} contains several lemmas, used in the proofs of Theorems  \ref{thmDqdv} and \ref{thmcpq}, 
about members of $\mathcal{B}$ that are multiples of  a natural number $q$. 
The proof of Theorem \ref{thmDqdv} is given in Section \ref{secpot2}.
In Section \ref{secpot3} we establish Theorem \ref{thmcpq} with the help of Corollary \ref{corDqdv}, which is a consequence of Theorem \ref{thmDqdv}.
Finally, in Section \ref{secpot1} we apply Theorems  \ref{thmDqdv} and \ref{thmcpq} to prove Theorem \ref{thmave}.

\section{Proof of Corollaries \ref{corpra}, \ref{cor2} and \ref{cortau}}\label{seccorpra}

\begin{lemma}\label{DqUB}
We have
\begin{equation*}
D(x,t) \ll \frac{x\log t}{\log xt} \qquad (x>1/t, t\ge 2),
\end{equation*}
\begin{equation*}
D(x,t) \gg \frac{x\log t}{\log xt} \qquad (x\ge 1, t\ge 2),
\end{equation*}
$$
 D(x/q,t) - D(q/t,t) \le D_q(x,t) \le D(x/q,qt)  \qquad (x\ge 0, t\ge 2, q\ge 1).
 $$ 
\end{lemma}
\begin{proof}
The first two estimates follow from \cite[Thm. 1]{Saias1}.

If $m\in \mathcal{D}(x/q,t)$ and $m>q/t$, then $mq \in \mathcal{D}_q(x,t)$. This shows that $D(x/q,t) - D(q/t,t) \le D_q(x,t)$.

If $n\in \mathcal{D}(x,t)$ and $q|n$, we write $n=qm$ and observe that $m \in \mathcal{D}(x/q,qt)$. 
Thus, $D_q(x,t) \le D(x/q,qt)$.
\end{proof}

\begin{proof}[Proof of Corollary \ref{corpra}]
We first show \eqref{cor1norm}.
For $1\le n \le x$, we have $\max(2,n)\le \theta(n)\le n f(n) \le n f(x)$.
Thus, $\mathcal{B}(x) \subset \mathcal{D}(x,f(x))$. If $f(x) \le x$, \eqref{eqnorm} yields
$$
\left| \left\{ n \in \mathcal{B}(x) : |\omega(n)-E(x,f(x))| > \frac{\xi(x)}{2} \sqrt{\log_2 x}\right\} \right|
\ll \frac{D(x,f(x))}{\xi(x)^2}.
$$
The assumption $\theta(n)\ge \max(2,n)$ implies $B(x) \gg x/\log x$, by \cite[Thm. 1.2]{PDD}.
By Lemma \ref{DqUB},
$$
\frac{D(x,f(x))}{\xi(x)^2} \ll \frac{x \log f(x)}{\xi(x)^2 \log x} \ll B(x) \frac{ \log f(x)}{\xi(x)^2}. 
 $$ 
 The result being trivial if $\log f(x) > \xi(x)^2$, we may assume $\log f(x) \le \xi(x)^2$, so that 
$$
|  E(x,f(x)) - C\log_2 x | = |(C-1)\log_2 f(x)| \le \frac{\xi(x)}{2} \sqrt{\log_2 x},
$$
for $x\ge x_0$. Thus, \eqref{cor1norm} holds if $f(x)\le x$. 
If $f(x)>x$, then $\xi(x)^2 \ge \log f(x) >\log x$, so that $|\omega(n)-C\log_2 x|>\xi(x) \sqrt{\log_2 x}$
implies $\omega(n)> \xi(x) > \sqrt{\log x}$. 
The result now follows from  Nicolas' Theorem \cite{Nic}, an asymptotic estimate for the quantity
$|\{n\le x: \Omega(n)=k\}|$, which easily implies 
\begin{equation}\label{Niceq}
|\{n\le x: \Omega(n)\ge y \log_2 x\}| \ll  \frac{x}{(\log x)^{y \log 2 -1}},
\end{equation}
uniformly for $x\ge 2$ and $y \ge 2+\delta$, for any fixed $\delta>0$. 

Next, we show that  \eqref{cor1norm} implies \eqref{cor1ave}.
From \eqref{Niceq} we have
\begin{equation}\label{largeomega}
|\{n\le x: \omega(n)\ge 6 \log_2 x\}|\le |\{n\le x: \Omega(n)\ge 6 \log_2 x\}| 
\ll \frac{B(x)}{\log^2 x}.
\end{equation}
Since $\omega(n)\le \Omega(n) \ll \log n$, the contribution to $\sum_{n\in \mathcal{B}(x)} \omega(n)$
from $n$ with $\omega(n)>6 \log_2 x$ is $\ll B(x)/\log x$, while 
the contribution from $n$ with $\omega(n)\le 6 \log_2 x$ and $|\omega(n)-C\log_2 x|>\xi(x) \sqrt{\log_2 x}$ is 
$$
\ll (6\log_2 x)  B(x)\frac{\log f(x)}{\xi(x)^2},
$$
by \eqref{cor1norm}. The contribution to $\sum_{n\in \mathcal{B}(x)} \omega(n)$ from $n$ with $\omega(n)\le 6 \log_2 x$ and $|\omega(n)-C\log_2 x|\le \xi(x) \sqrt{\log_2 x}$ is
$$
B(x)\left(1+O\left(\frac{\log f(x)}{\xi(x)^2} + \frac{1}{\log^2 x} \right)\right) 
C\log_2 x \left(1+O\left(\frac{\xi(x)}{\sqrt{\log_2 x}}\right)\right).
$$
If $\log f(x) \le \log_2 x$, \eqref{cor1ave} now follows with $\xi(x)=(\log f(x))^{1/3} (\log_2 x)^{1/6}$.
If $\log f(x) > \log_2 x$, \eqref{cor1ave} follows directly from \eqref{largeomega}. 
The argument works the same with $\Omega(n)$ in place of $\omega(n)$. 
\end{proof}

\begin{proof}[Proof of Corollary \ref{cor2}]
Assume $\max(2,n) \le \theta(n) \ll n (\log 2n)^{o(1)}$. Define $f(x) = \max_{n\le x} \theta(n)/n$, so that 
$f$ is non-decreasing and $f(x)=(\log x)^{o(1)}$ as $x\to \infty$, that is $\log f(x) = o(\log_2 x)$. 
The relation \eqref{cor2ave} follows from \eqref{cor1ave}. 
Choosing $\xi(x)=(\log f(x) \log_2 x)^{1/4}$ in  \eqref{cor1norm}  yields  \eqref{cor2norm}. 
\end{proof}

\begin{lemma}\label{lemOmTau}
Let $\varepsilon >0$. For $2\le \alpha \le 4-\varepsilon$ we have
$$
\sum_{n\le x \atop \Omega(n) \ge \alpha \log_2 x} \tau(n) \ll x (\log x)^{\alpha(\log 2 - \log \alpha +1)-1}.
$$
\end{lemma}

\begin{proof}
This is a variation of Exercise 05 in \cite{HT}.
Write $y^{\Omega(n)} = \sum_{d|n} f(d)$, so that $f(n)$ is multiplicative and 
$f(p^k)=y^k(1-1/y)$ for $k\ge 1$, by M\"{o}bius inversion. For $0\le y \le 2-\varepsilon$, 
\begin{equation*}
\begin{split}
\sum_{n\le x} \tau(n)  y^{\Omega(n)} & = \sum_{n\le x}\tau(n) \sum_{d|n} f(d) 
\le  \sum_{d\le x} f(d) \tau(d) \sum_{m\le x/d}\tau(m) \\
& \le x \log x \sum_{d\le x} f(d)\tau(d) /d 
\le  x \log x  \sum_{P^+(d)\le x} f(d)\tau(d) /d \\
& = x \log x \prod_{ p \le x} \sum_{k\ge 0} f(p^k)\tau(p^k)/p^k 
\ll x (\log x)^{2y-1}.
\end{split}
\end{equation*}
If $1\le y \le 2-\varepsilon$, we get 
$$
\sum_{n\le x \atop \Omega(n) \ge \alpha \log_2 x} \tau(n) y^{ \alpha \log_2 x} \ll x(\log x)^{2y-1}.
$$
The result now follows with $y=\alpha/2$.
\end{proof}

\begin{proof}[Proof of Corollary \ref{cortau}]
Since  $2^{\omega(n)} \le \tau(n) \le 2^{\Omega(n)}$ for all $n\ge 1$, 
the estimate \eqref{cortaunorm} and the lower bound in \eqref{cortauave} follow at once from \eqref{cor2norm}.
For the upper bound in \eqref{cortauave}, we write 
$$
\sum_{n\in \mathcal{B}(x) \atop \Omega(n) \le e \log_2 x} \tau(n) \le B(x) 2^{e \log_2 x} = B(x) (\log x)^{e \log 2}
$$
and
$$
\sum_{n\in \mathcal{B}(x) \atop \Omega(n) \ge e \log_2 x} \tau(n) 
\le \sum_{n\le x \atop \Omega(n) \ge e \log_2 x} \tau(n) 
\ll  x (\log x)^{e \log 2 -1} \ll B(x) (\log x)^{e \log 2},
$$
by Lemma \ref{lemOmTau} with $\alpha = e$. 
\end{proof}

\section{Multiples of $q$ in $\mathcal{B}$}\label{secmult}

In this section we develop some general identities for sets $\mathcal{B}$, defined by \eqref{Bdef}, with
\begin{equation}\label{thetacond}
\theta : \mathbb{N} \to \mathbb{R}\cup \{\infty\}, \quad \theta(1)\ge 2, \quad \theta(n)\ge P^+(n) \quad (n\ge 2),
\end{equation}
where $P^+(n)$ denotes the largest prime factor of $n$. 
Let
$$
\Phi(x,y)=1_{x\ge 1}+|\{2\le n\le x : P^-(n)>y\}|,
$$
where $P^-(n)$ denotes the smallest prime factor of $n$. 
Let
$$
\psi(n) := 
\begin{cases} 1 &\mbox{if } n \in \mathcal{B}\\
 0 & \mbox{else.} \end{cases}
$$
and define
$$
\lambda_n(s) := \frac{\psi(n)}{n^s} \prod_{p\le \theta(n)}\left(1-\frac{1}{p^s}\right), \quad \lambda_n:=\lambda_n(1),
$$
$$
\mu_n(s) := \sum_{p\le \theta(n)} \frac{\log p}{p^s-1} - \log n, \quad \mu_n:=\mu_n(1).
$$

\begin{lemma}
Let $\theta$ satisfy \eqref{thetacond} and 
let $q_1\le q_2 \le \ldots \le q_k$ be primes.
For $x\ge 0$, 
\begin{equation}\label{phi0}
\sum_{n\ge 1} \psi(n)\Phi\left(\frac{x}{n},\theta(n)\right) =\lfloor x \rfloor
\end{equation}
and
\begin{equation}\label{phik}
\sum_{n\ge 1 \atop q_1\cdots q_k |n} \psi(n)\Phi\left(\frac{x}{n},\theta(n)\right)
= \sum_{\theta(n)\ge q_k \atop q_1\cdots q_{k-1} |n}\psi(n)\Phi\left(\frac{x}{nq_k},\theta(n)\right).
\end{equation}
\end{lemma}

\begin{proof}
The relation \eqref{phi0} is \cite[Lemma 3]{SPA}. We will show \eqref{phik}.
Every $m\in q_1\cdots q_k \mathbb{N}$ factors uniquely as $m=nr$ where $n\in \mathcal{B}$ and $P^-(r)>\theta(n)$ if $r>1$.
If $q_1 \nmid n$ then $\theta(n)<q_1$. 
If $q_1|n$, let $j$ be the largest index such that $q_1\cdots q_j|n$, so that $q_{j+1}>\theta(n)$ if $j<k$. 
We count all integer multiples of $q_1\cdots q_k$ up to $x$ according to 
$j$ and $n$:
\begin{multline}\label{phisum}
\left\lfloor \frac{x}{q_1\cdots q_k}\right\rfloor = 
\sum_{\theta(n)<q_1}\psi(n) \Phi\left(\frac{x}{n q_1\cdots q_k},\theta(n)\right)\\
+\sum_{j=1}^{k-1} \sum_{\theta(n)<q_{j+1} \atop q_1\cdots q_{j}|n} \psi(n) \Phi\left(\frac{x }{n q_{j+1}\cdots q_k},\theta(n)\right)
+ \sum_{q_1\cdots q_{k}|n}\psi(n) \Phi\left(\frac{x}{n},\theta(n)\right).
\end{multline}
We can now establish \eqref{phik} by induction on $k$. 
When $k=1$, \eqref{phisum} and \eqref{phi0} yield \eqref{phik}. 
For the inductive step, we write the inner sum of \eqref{phisum} as
$$
\sum_{\theta(n)<q_{j+1} \atop q_1\cdots q_{j}|n}  = \sum_{ q_1\cdots q_{j}|n} - \sum_{\theta(n)\ge q_{j+1} \atop q_1\cdots q_{j}|n}
$$
and use the inductive hypothesis on the sum $ \sum_{ q_1\cdots q_{j}|n} $ to get
\begin{multline*}
 \sum_{\theta(n)<q_{j+1} \atop q_1\cdots q_{j}|n}  \psi(n) \Phi\left(\frac{x }{n q_{j+1}\cdots q_k},\theta(n)\right)
   = \sum_{ \theta(n)\ge q_{j} \atop q_1\cdots q_{j-1}|n} \psi(n) \Phi\left(\frac{x }{n q_{j}\cdots q_k},\theta(n)\right)\\
  - \sum_{\theta(n)\ge q_{j+1} \atop q_1\cdots q_{j}|n}  \psi(n) \Phi\left(\frac{x }{n q_{j+1}\cdots q_k},\theta(n)\right).
\end{multline*}
Thus, the sum over $j$ in \eqref{phisum} is telescoping and the result follows from \eqref{phi0}.
\end{proof}

\begin{lemma}\label{lemlambda}
Let $\theta$ satisfy \eqref{thetacond} and 
let $q_1\le q_2 \le \ldots \le q_k$ be primes. 
For $\re(s)>1$ we have
\begin{equation}\label{lambda0}
\sum_{n\ge 1 } \lambda_n(s) = 1
\end{equation}
and
\begin{equation}\label{lambdak}
\sum_{n\ge 1 \atop q_1\cdots q_k |n} \lambda_n(s) = \frac{1}{q_k^s} \sum_{\theta(n)\ge q_k \atop q_1\cdots q_{k-1} |n} \lambda_n(s).
\end{equation}
Both relations hold at $s=1$ if $B(x)=o(x)$. 
\end{lemma}

\begin{proof}
The relation \eqref{lambda0} is \cite[Lemma 1]{CFAE} when $\re(s)>1$ and \cite[Theorem 1]{SPA} when $s=1$. 
The proof of \eqref{lambdak} mirrors that of \eqref{phik}. We first assume $\re(s)>1$. 
Every $m\in q_1\cdots q_k \mathbb{N}$ factors uniquely as $m=nr$ where $n\in \mathcal{B}$ and $P^-(r)>\theta(n)$ if
$r>1$.
If $q_1 \nmid n$ then $\theta(n)<q_1$. 
If $q_1|n$, let $j$ be the largest index such that $q_1\cdots q_j|n$, so that $q_{j+1}>\theta(n)$ if $j<k$. 
We rearrange the terms of the Dirichlet series $\sum_{ q_1\cdots q_k |m } m^{-s}$ according to $n$ and $j$. 
After dividing by $\zeta(s)$, this shows that, for $\re(s) >1$,
\begin{equation}\label{zetaid}
\frac{1}{(q_1\cdots q_k)^s} =\sum_{\theta(n)<q_1} \frac{\lambda_n(s) }{(q_1\cdots q_k)^s} 
+\sum_{j=1}^{k-1} \sum_{\theta(n)<q_{j+1} \atop q_1\cdots q_{j}|n} \frac{\lambda_n(s) }{(q_{j+1}\cdots q_k)^s}
+ \sum_{q_1\cdots q_{k}|n}\lambda_n(s) .
\end{equation}
We establish \eqref{lambdak} by induction on $k$.
When $k=1$, the result follows from applying \eqref{lambda0} to the first sum of \eqref{zetaid}. 
For the inductive step, note that the inner sum in \eqref{zetaid} is
\begin{equation*} 
\begin{split}
 \sum_{\theta(n)<q_{j+1} \atop q_1\cdots q_{j}|n} \frac{\lambda_n(s)}{(q_{j+1}\cdots q_k)^s}
  & = \sum_{ q_1\cdots q_{j}|n} \frac{\lambda_n(s)}{(q_{j+1}\cdots q_k)^s}
  - \sum_{\theta(n)\ge q_{j+1} \atop q_1\cdots q_{j}|n} \frac{\lambda_n(s)}{(q_{j+1}\cdots q_k)^s} \\
 & = \sum_{ \theta(n)\ge q_{j} \atop q_1\cdots q_{j-1}|n} \frac{\lambda_n(s)}{(q_{j}\cdots q_k)^s}
  - \sum_{\theta(n)\ge q_{j+1} \atop q_1\cdots q_{j}|n} \frac{\lambda_n(s)}{(q_{j+1}\cdots q_k)^s}, \\
\end{split}
\end{equation*}
by the inductive hypothesis. Thus, the sum over $j$ in \eqref{zetaid} is a telescoping sum and the result follows from \eqref{lambda0}. 

If $B(x)=o(x)$, the validity of \eqref{lambdak} at $s=1$ follows from \eqref{phik}, in much the same way that 
the validity of \eqref{lambda0} at $s=1$ follows from \eqref{phi0}, which is demonstrated in the proof of \cite[Thm. 1]{SPA}.
\end{proof}

\begin{lemma}
Let $\theta$ satisfy \eqref{thetacond} and 
let $q_1\le q_2 \le \ldots \le q_k$ be primes. For $\re(s)>1$ we have
\begin{equation}\label{mu0}
\sum_{n\ge 1 } \lambda_n(s)\mu_n(s) = 0
\end{equation}
and
\begin{equation}\label{muk}
\sum_{n\ge 1 \atop q_1\cdots q_k |n} \lambda_n(s)\mu_n(s) 
= \frac{1}{q_k^s} \sum_{\theta(n)\ge q_k \atop q_1\cdots q_{k-1} |n} \lambda_n(s)
\bigl(\mu_n(s) - \log q_k\bigr).
\end{equation}
\end{lemma}
\begin{proof}
Differentiate \eqref{lambda0} and \eqref{lambdak} with respect to $s$. 
\end{proof}

\section{Proof of Theorem \ref{thmDqdv}}\label{secpot2}

The following estimate for $\Phi(x,y)$, which differs from the one we used in \cite{PDD}, 
simplifies the proof of Theorem  \ref{thmDqdv}.

\begin{lemma}\label{PhiLemma}
Uniformly, for $x\ge 0$, $y\ge 2$, we have
\begin{equation*}
\begin{split}
\Phi(x,y) &= 1_{x\ge 1} + x \prod_{p\le y}\left(1-\frac{1}{p}\right)
+\frac{x}{\log y}\! \left\{w(u)-e^{-\gamma}\! -\left.\frac{y}{x}\right|_{x\ge y} 
\!\!+ O\left(\frac{e^{-u/3}}{\log y}\right)\! \right\}\\
&= 1_{x\ge 1} + x \prod_{p\le y}\left(1-\frac{1}{p}\right)
+O\left(\frac{xe^{-u/3}}{\log y}\right),
\end{split}
\end{equation*}
where $u=\frac{\log \max(1,x)}{\log y}$ and $w(u)$ is Buchstab's function. 
\end{lemma}

\begin{proof}
The second estimate follows from the first, since $w(u)-e^{-\gamma}\ll e^{-u}$ and,  if $x\ge y\ge 2$,  
then $y/x \ll  e^{-u/2}$. 

When $x\ge y \ge 2$ and $\log y \ge (\log_2 x)^2$, 
the first estimate follows from combining equations (49), (52), (59) and (60) of \cite[Sec. III.6]{Ten},
with equation (6) of \cite[Sec. III.5]{Ten},
where we estimate the integral in (52) as
$$
\int_0^\infty |w'(u-v)|y^{-v} dv \ll \int_0^\infty e^{-(u-v)/2} y^{-v} dv=\frac{e^{-u/2}}{\log y -1/2}
 \asymp \frac{e^{-u/2}}{\log y}.
$$

When $x\ge y \ge 2$ and $\log y < (\log_2 x)^2$, then $u \gg \sqrt{\log x}$ and the result follows from
\cite[Thm. III.6.1 and Thm. III.5.1]{Ten}.

When $x<y$, then $\Phi(x,y)=1_{x\ge 1}$ and $w(u)=0$, so that the result follows from Mertens' formula 
\cite[Thm. I.1.11]{Ten}. 
\end{proof}

\begin{lemma}\label{Laplace}
Assume $B(x)=B_t(x)$ is the counting function of a set $\mathcal{B}_t \subset \mathbb{N}$ that depends on the parameter $t$.
Assume
\begin{equation*}
B(x) = x \int_1^\infty \frac{B(y)}{y^2 \log yt} \left(e^{-\gamma}-   w\left(\frac{\log x/y}{\log yt}\right) \right)dy
+ R(x)\quad (x\ge 1),
\end{equation*}
such that the integrals  
$$
\alpha_{t} :=  e^{-\gamma} \int_1^\infty \frac{B(y)}{y^2 \log yt} dy, \quad
\beta_{t} := \frac{-1}{\log t} \int_1^\infty R(y) \frac{dy}{y^2}
$$
converge.
Then
\begin{equation*}
B(x) = x \eta_{t} d(v) + O\Bigl\{1+ x \beta_{t}(v+1)^{-3.03} + R(x)+I(x)+J(x)\Bigr\},
\end{equation*}
where $v= \log x / \log t$, $d(v)$ is given by \eqref{dinteq}, $\eta_{t} = \alpha_{t} + \beta_{t}$,  
$$
I(x)=\frac{x}{\log xt} \int_x^\infty R(y) \frac{dy}{y^2}, \quad 
J(x)  =\frac{x}{(\log xt)^{3.03}} \int_1^x R(y) (\log yt)^{2.03} \frac{dy}{y^2}.
$$
\end{lemma}
\begin{proof}
We follow the second half of the proof of \cite[Thm. 1.3]{PDD}.
The only modification needed is the use of the improved estimates
\begin{equation}\label{dest}
(v+1) d(v) = C + O((v+1)^{-2.03}), \quad (v\ge 0),
\end{equation}
and
\begin{equation}\label{dpest}
(v+1)^2 d'(v) = -C + O((v+1)^{-2.03}), \quad (v\ge 0).
\end{equation}
The estimate \eqref{dest} is a consequence of \cite[Cor. 6]{IDD3}, while 
\eqref{dpest} follows from inserting \eqref{dest} in the proof of \cite[Cor. 5]{IDD3}.
In \cite{PDD}, we used slightly weaker estimates for simplicity, with an exponent of $-2$ instead of $-2.03$ in the error terms.
In the proof of Theorem \ref{thmDqdv}, the improved exponent will save a factor of $\log_2 x$ (when estimating the contribution
from $R_2(x)$ to $J(x)$). 
\end{proof}

For $n\in \mathbb{N}$ with prime factorization $n=p_1 \cdots p_k$, where $p_1\le \ldots \le p_k$, define 
$$
F(n) := \max_{1\le j \le k} p_j^2 p_{j+1}\cdots p_k.
$$

\begin{lemma}\label{lemDzero}
If $x<\max(m,F(m)/t)$, then $D_m(x,t)=0$. 
\end{lemma}
\begin{proof}
Note that $n \in \mathcal{D}_m(x,t)$ if and only if $n\le x$, $m|n$ and $F(n)\le nt$. 
Also, $m|n$ implies $F(m)\le F(n)$.
Thus, if $D_m(x,t)\neq 0$ and $n \in \mathcal{D}_m(x,t)$, then $m\le n \le x$ and $F(m)\le F(n) \le nt \le xt$,
so $x\ge  \max(m,F(m)/t)$.
\end{proof}

\begin{lemma}\label{lemprodUB}
Let $n\ge 2$ with prime factorization $n=p_1\cdots p_k$, $p_1 \le \ldots \le p_k$. If $F(n)\le x$, then 
$p_{k-j+1}\cdots p_k \le x^{1-2^{-j}}$ for $1\le j \le k$. 
\end{lemma}

\begin{proof}
We use induction on $j$. When $j=1$, the claim is that $p_k \le x^{1/2}$, which follows from $p_k^2 \le F(n) \le x$, for all $k\ge 1$.
Assume now that the claim is correct for some $j \in \mathbb{N}$ and all $k\ge j$. Let $k\ge j+1$. 
We have $p_{k-j}^2 p_{k-j+1}\cdots p_k \le F(n) \le x$. Thus, if $p_{k-j} p_{k-j+1}\cdots p_k \ge x^{1-2^{-j-1}}$, then
$p_{k-j} \le  x^{2^{-j-1}}$. By the inductive hypothesis,
$$
p_{k-j} (p_{k-j+1}\cdots p_k) \le  x^{2^{-j-1}} x^{1-2^{-j}}= x^{1-2^{-j-1}},
$$
for all $k\ge j+1$. 
\end{proof}

\begin{lemma}\label{1lem}
Let $k\ge 0$ be fixed. For $ m\ge 1$ with $\Omega(m)=k$, $t\ge 2$ and $x\ge \max(m, F(m)/t)$,  we have
$$
 \frac{x \log t \log 2m}{m \log xt \log 2x} \gg_k 1.
$$
\end{lemma}
\begin{proof}
This is obvious if $m=1$ or if $t\ge x^{2^{-k}}$. If $m\ge 2$ and  $t<x^{2^{-k}}$, then $F(m)\le xt$ and Lemma \ref{lemprodUB} imply
$m\le (xt)^{1-2^{-k}}< x^{1-4^{-k}} \ll_k x/(\log xt)^2$, from which the claim follows. 
\end{proof}

\begin{proof}[Proof of Theorem \ref{thmDqdv}]
In the remainder of this paper, we write $D_q(x)$ for $D_q(x,t)$ and $D(x)$ for $D(x,t)$. 
We will show by induction on $k\ge 0$ that, for $k=\Omega(q)$, the estimates \eqref{eqDqdv0}, \eqref{eqDqdv} and 
\eqref{etaeq}  hold  and 
that, for primes $r$ with $r\ge P^+(q)$, we have  
\begin{equation}\label{indhyp}
\sum_{n \in \mathcal{D}_q \atop n\ge r/t}  \Phi(x/rn,nt)
 -\frac{x}{r} \sum_{n \in \mathcal{D}_q \atop n\ge r/t}\lambda_n
\ll 1+ \frac{x \log t \log qr \log qrt}{qr \log^3 xt}=:R_2(x,qr),
\end{equation}
for $x\ge \max(qr,F(qr)/t)$.
Note that \eqref{eqDqdv} and \eqref{dest} imply
\begin{equation}\label{indUB}
D_q(x) \ll_k  x d(v) \eta_{q,t} \asymp_k \frac{x \log t}{q \log xt}, \quad (x\ge 1, q\ge 1, t\ge 2),
\end{equation}
since $D_q(x)=0$ if $q>x$. 

When $k=0$, $q=1$, \eqref{eqDqdv0} is \cite[Eq. (13)]{PDD}, \eqref{eqDqdv} is \cite[Thm. 1.3]{PDD}
and \eqref{etaeq} is \cite[Eq. (6)]{PDD}. 
To show \eqref{indhyp} for $k=0$ and $q=1$, assume that $r$ is prime and $x\ge \max(r,r^2/t)$. 
Equations \eqref{phi0} and \eqref{lambda0} show that 
\begin{equation}\label{lastrow}
\begin{split}
& \sum_{n \in \mathcal{D} \atop n\ge r/t}  \Phi(x/rn,nt)
 - \frac{x}{r}\sum_{n \in \mathcal{D} \atop n\ge r/t}  \lambda_n \\
 & = -\{x/r\} -\sum_{n \in \mathcal{D} \atop n< r/t}  \Phi(x/rn,nt)
 +\frac{x}{r}\sum_{n \in \mathcal{D} \atop n< r/t}\lambda_n \\
 & \ll 1+ D(r/t) 
+ \sum_{n\in \mathcal{D}\atop n< r/t} \frac{x}{rn \log nt} 
\exp\left(-\frac{\log xt/r}{3\log nt}\right),
\end{split}
\end{equation}
by the second estimate in Lemma \ref{PhiLemma}.
We have $ D(r/t)\ll r \log t /(t \log r)$ by Lemma \ref{DqUB}, so $ D(r/t) \ll R_2(x,r)$ follows from 
$$
\frac{r}{\log^2  r \log rt } \ll \frac{xt/r}{\log^3 xt} \asymp \frac{xt/r}{\log^3 (xt/r)},
$$
since $xt/r \ge \sqrt{xt}$. This holds because $r \le xt/r$. 
Finally, the last sum in \eqref{lastrow} is $\ll R_2(x,r)$ by Lemma \ref{DqUB}.
Thus, \eqref{indhyp} holds for $k=0$.

For the inductive step, 
assume that \eqref{eqDqdv} (and hence \eqref{indUB}) and \eqref{etaeq} hold for 
$q\in \mathbb{N}$ with $\Omega(q)= k$ for some $k\ge 0$.
If $k\ge 1$, assume that \eqref{indhyp} holds for $\Omega(q) = k-1$. 
Let $r$ be a prime with $r\ge P^+(q)$ and write $m=qr$.
We note that in the remainder of this proof, all implied constants in the $\ll$ and big-O notation may depend on $k$.

We estimate the first sum in \eqref{phik} with Lemma \ref{PhiLemma} and apply \eqref{lambdak} to get
\begin{equation}\label{BigEqk}
\begin{split}
 D_{m}(x) & + \sum_{n\in \mathcal{D}_m} \frac{x}{n \log nt} 
\left\{w\left(\frac{\log x/n}{\log nt}\right)-e^{-\gamma}-\left.\frac{n^2 t}{x}\right|_{n^2\le  \frac{x}{t}} 
\! + O\left(\frac{e^{-\frac{\log x/n}{3\log nt }}}{\log nt}\right)\right\}\\
&=\sum_{n \in \mathcal{D}_q \atop n\ge r/t}  \Phi(x/rn,nt)
 - \frac{x}{r} \sum_{n \in \mathcal{D}_q \atop n\ge r/t} \lambda_n .
\end{split}
\end{equation}
The contribution from the last two terms in the first sum in \eqref{BigEqk} is
$$
\ll \tilde{R}_1(x):= \frac{x \log mt}{m (\log xt)^2},
$$
by Lemma \ref{DqUB}. In the second application of this argument we will be able to replace $ \tilde{R}_1(x)$
by the smaller
$$
R_1(x) :=\frac{x\log t}{m (\log xt)^2}.
$$
 The error from applying Abel summation to the remaining terms of the first sum in \eqref{BigEqk} is also 
$\ll \tilde{R}_1(x)$, since $w(u)-e^{-\gamma} \ll e^{-u}$ and $w'(u)\ll e^{-u}$. 
Thus,
\begin{equation}\label{BigEq2k}
\begin{split}
 D_m(x) = & x \int_1^\infty \frac{D_m(y)}{y^2 \log yt} \left(e^{-\gamma}-   w\left(\frac{\log x/y}{\log yt}\right) \right)dy + O(\tilde{R}_1(x)) \\
&+\sum_{n \in \mathcal{D}_q \atop n\ge r/t}  \Phi(x/rn,nt)
 -\frac{x}{r} \sum_{n \in \mathcal{D}_q \atop n\ge r/t} \lambda_n.
\end{split}
\end{equation}

If $x/r<\max(q,F(q)/t,r/t)$, that is $x<\max(m,F(m)/t)$, then Lemma \ref{lemDzero} shows that the first sum in \eqref{BigEq2k} vanishes, while the second sum is
$$
\ll  \frac{x}{m}\min\left(1,\frac{\log t}{\log r}\right)\ll    \frac{x}{m}\min\left(1,\frac{\log t}{\log m}\right),
$$
by the inductive hypothesis \eqref{indUB} and since $\log r \le \log m \le \log r^{k+1} \ll_k \log r$.
Define
$$
R_2(x):=
  \begin{cases} 
     \frac{x}{m}\min\left(1,\frac{\log t}{\log m}\right) & \text{if } x <\max(m,F(m)/t) \\
   1+\frac{x \log t \log m \log mt}{m\log^3 xt}       & \text{if } x \ge \max(m,F(m)/t).
  \end{cases}
$$
Thus,
\begin{equation}\label{inteq}
D_m(x) = x \int_1^\infty \frac{D_m(y)}{y^2 \log yt} \left(e^{-\gamma}-   w\left(\frac{\log x/y}{\log yt}\right) \right)dy
+ R(x)
\end{equation}
holds with  $R(x) \ll \tilde{R}_1(x) + R_2(x)$ when $x<\max(m,F(m)/t)$. 

Assume now that $x\ge \max(m,F(m)/t)$.
If $k=0$ and $q=1$, we have already shown that the last row of \eqref{BigEq2k} is $\ll R_2(x)$. 
If $k\ge 1$, write $q=u v$ where $v=P^+(q)$. 
Equations \eqref{phik} and \eqref{lambdak} show that the last row of \eqref{BigEq2k} equals

\begin{equation*}
\begin{split}
=&\sum_{n \in \mathcal{D}_u \atop n\ge v /t}  \Phi(x/nrv,nt)
-\frac{x}{rv} \sum_{n \in \mathcal{D}_u \atop n\ge v /t} \lambda_n
-\sum_{n \in \mathcal{D}_q \atop n< r /t}  \Phi(x/nr,nt)
+\frac{x}{r} \sum_{n \in \mathcal{D}_q \atop n< r /t} \lambda_n \\
=& S-T - (U-V),
\end{split}
\end{equation*}
say. 
Now $x\ge \max(m,F(m)/t)$ implies $x/r \ge \max(q,F(q)/t,r/t)$. 
The inductive hypothesis \eqref{indhyp} yields
\begin{equation*}
S-T  \ll 1+\frac{(x/r) \log t \log q \log qt}{q \log^3 (xt/r)}\ll 1+ \frac{x\log t \log m \log mt}{m \log^3 xt}=R_2(x),
\end{equation*}
since $xt/r \ge \sqrt{xt}$. 
The second estimate in Lemma \ref{PhiLemma} and \eqref{indUB} show that
\begin{equation*}
\begin{split}
U-V  & \ll D_q(r/t) +
\sum_{n \in \mathcal{D}_q \atop n< r /t} \frac{x}{n r \log nt}\exp\left(-\frac{\log \frac{x}{nr}}{3\log nt}\right)\\
& \ll \frac{r^2 \log t}{mt \log r} +R_2(x)\ll R_2(x) ,
\end{split}
\end{equation*}
where the last assertion is implied by
$$
 \frac{r}{\log r \log m \log mt} \ll \frac{xt/r}{\log^3 (xt/r)},
$$
which holds because $mt > m\ge r$ and $r\le xt/r$.
As the last row of \eqref{BigEq2k} is $\ll R_2(x)$, 
we have established that \eqref{indhyp} holds for $\Omega(q)\le k$
and that \eqref{inteq} holds with  $R(x) \ll \tilde{R}_1(x) + R_2(x)$.

We need to estimate $I(x)$ and $J(x)$ from Lemma \ref{Laplace}. 
For this purpose we may assume that $x\ge \max(m,F(m)/t)$. If not, 
Lemma \ref{lemDzero} shows that $D_m(x)=0$, so that \eqref{eqDqdv0} holds
since the main term in \eqref{eqDqdv0} is absorbed by the error terms. 
We find that 
$I(x)+ J(x) \ll \tilde{R}_1(x)+R_2(x)$ and $\beta_t \ll \frac{\log m}{m\log t}$ so that $x\beta_t (v+1)^{-3} \ll R_2(x)$. 

The conclusion of Lemma \ref{Laplace} is that 
\begin{equation}\label{conclusion}
D_m(x) = x \eta_{m,t} d(v) + O\bigl( \tilde{R}_1(x)+R_2(x)\bigr).
\end{equation}
The lower bound in Lemma \ref{DqUB} yields $ \eta_{m,t}\gg m^{-1}$. 
 Lemma \ref{cUBlem} shows that $\eta_{m,t}\ll_k m^{-1}$. 
 Together with \eqref{conclusion}, Lemmas \ref{lemDzero} and \ref{1lem}, this implies
\begin{equation}\label{Dupper}
D_m(x)\ll_k \frac{x\log t}{m\log xt}, \quad (x\ge 1, t\ge 2, \Omega(m)\le k+1).
\end{equation}
Running through this proof a second time  with this upper bound replacing the one in Lemma \ref{DqUB},
we can replace $\tilde{R}_1(x)$ by $R_1(x)$ 
to obtain
\begin{equation*}
D_m(x) = x \eta_{m,t} d(v) + O\bigl(R_1(x)+R_2(x)\bigr),
\end{equation*}
where $1/m \ll  \eta_{m,t} \ll_k 1/m $. 
This shows that \eqref{eqDqdv0} and \eqref{etaeq} hold with $q$ replaced by $m=qr$. 

To see that \eqref{eqDqdv0} implies \eqref{eqDqdv}, note that
the term $O(1)$ is acceptable by Lemma \ref{1lem}, provided $x\ge \max(m,F(m)/t)$. 
If $x<  \max(m,F(m)/t)$, then $D_m(x)=0$ and $x<F(m)\le m^2$, so \eqref{eqDqdv} (with $q$ replaced by $m$)
follows from $\log x \le 2 \log m$. 
This completes the proof 
of Theorem \ref{thmDqdv}.
\end{proof}

\begin{lemma}\label{cUBlem}
Assume that \eqref{conclusion} holds for $m$ with $\Omega(m)\le k+1$,
and that $\eta_{q,t} \ll_k q^{-1}$ for $\Omega(q)\le k$. Then $\eta_{m,t} \ll_k m^{-1} $ for $\Omega(m)=k+1$. 
\end{lemma}

\begin{proof}
Let $m=qr$ where $r$ is prime and $r\ge P^+(q)$. Equation \eqref{muk} yields
\begin{equation*}
 \sum_{n\ge 1 \atop qr|n} \lambda_n(s) \mu_n(s)
= \frac{1}{r^s} \sum_{n\ge r/t \atop q|n} \lambda_n(s) 
 \bigl(\mu_n(s)-\log r \bigr).
\end{equation*}
As in \cite{CFAE}, we let $s=1+1/\log^2 N$ and split the sum on the left according to $n\le N$ and $n>N$. 
As $N\to \infty$, the contribution from $n\le N$ converges to (see \cite[Lemma 3]{CFAE})
$\sum_{ qr|n} \lambda_n \mu_n$,
while the contribution from $n>N$ converges to (see \cite[Lemma 4]{CFAE})
$
-c_{qr} (1-e^{-\gamma}), 
$
by \eqref{conclusion}, where $\eta_{qr,t}=C c_{qr} \log t$. 
Applying the same reasoning to the sum on the right-hand side, we get 
$$
 \sum_{n\ge 1 \atop qr|n} \lambda_n \mu_n -c_{qr} (1-e^{-\gamma})= \frac{1}{r} \Bigl(\sum_{n\ge r/t \atop q|n} \lambda_n \bigl(\mu_n-\log r \bigr) - c_q(1-e^{-\gamma}) \Bigr).
 $$
 Since $\mu_n \ll \log t $ , we obtain
$$
 r c_{qr} \ll S \log t + S \log r+ c_q,
\quad 
S:=  r \sum_{n\ge 1 \atop qr|n} \lambda_n =\sum_{n\ge r/t \atop q|n} \lambda_n ,
$$
by \eqref{lambdak}.
Now $\eta_{q,t}\ll q^{-1}$ and \eqref{Dupper} holds with $q$ in place of $m$. Thus,
$c_q \ll q^{-1}\log t$ and 
$
S= \sum_{n\ge r/t \atop q|n} \lambda_n \ll q^{-1} \min(1,\log t/ \log r).
$
This shows that 
$
r c_{qr} \ll q^{-1} \log t
$,
which is the desired result. 
\end{proof}

\section{Proof of Theorem \ref{thmcpq}}\label{secpot3}

\begin{proof}[Proof of \eqref{c1tasymp}, \eqref{cqtasymp1} and \eqref{cqtasymp2}]
The estimate \eqref{c1tasymp} is \cite[Cor. 3]{CFAE}.
Equation \eqref{muk}, with $k=1$ and $q_1=q$ prime,  is
\begin{equation}\label{Derivative}
 \sum_{n\ge 1 \atop q|n} \lambda_n(s) \mu_n(s)
= \frac{1}{q^s} \sum_{n\ge q/t} \lambda_n(s) 
 \bigl(\mu_n(s)-\log q \bigr).
\end{equation}
As in \cite{CFAE}, we let $s=1+1/\log^2 N$ and split the sum on the left according to $n\le N$ and $n>N$. 
As $N\to \infty$, the contribution from $n\le N$ converges to (see \cite[Lemma 3]{CFAE})
$\sum_{n\ge 1 \atop q|n} \lambda_n \mu_n$,
while the contribution from $n>N$ converges to (see \cite[Lemma 4]{CFAE})
$
-c_q (1-e^{-\gamma}), 
$
by Corollary \ref{corDqdv}.
Applying the same reasoning to the right-hand side of  \eqref{Derivative}, we get 
$$
 \sum_{n\ge 1 \atop q|n} \lambda_n \mu_n -c_{q} (1-e^{-\gamma})= \frac{1}{q} \Bigl(\sum_{n\ge q/t } \lambda_n \bigl(\mu_n-\log q \bigr) - c_\theta(1-e^{-\gamma}) \Bigr).
 $$
With the estimate  \cite[Lemma 13]{CFAE}
 \begin{equation}\label{muapprox}
\mu_n = \log t -\gamma + O\left(e^{-\sqrt{\log nt}}\right),
\end{equation}
we obtain
\begin{equation}\label{cqS}
 q c_{q} (1-e^{-\gamma})= S\log q + c_\theta (1-e^{-\gamma}) 
  +O\left(S e^{-\sqrt{\log \max(q,t)}}\right),
 \end{equation}
 where
 $$
S:= \sum_{n\ge q/t } \lambda_n = q \sum_{n\ge 1 \atop q|n} \lambda_n,
$$
by \eqref{lambdak}.
If $q \le t$, then $S=1$ by \eqref{lambda0}, so that \eqref{cqtasymp2} follows from \eqref{cqS}.

In general, we have 
\begin{equation}\label{Sest}
S= \frac{e^{-\gamma} c_\theta}{\log q} \left(1+O\left(\frac{1}{\log q} + \frac{\log^2 t}{\log^2 q}\right)\right), 
\quad (q\ge 2, \, t\ge 2).
\end{equation}
If $q\le t$, this holds because $S=1$ and $c_\theta \asymp \log t$. 
If $q\ge t$, \eqref{Sest} follows from applying Abel summation to $S= \sum_{n\ge q/t } \lambda_n$, together with Mertens' formula and Corollary \ref{corDqdv} (with $q=1$).
Combining \eqref{Sest} with \eqref{cqS} proves \eqref{cqtasymp1}. 
\end{proof}

\begin{proof}[Proof of \eqref{cpqtasymp1}, \eqref{cpqtasymp2} and \eqref{cpqtasymp3}]
Equation \eqref{muk}, with $k=2$ and $q_1=p \le q_2=q$  is
\begin{equation}\label{Derivative2}
 \sum_{n\ge 1 \atop pq|n} \lambda_n(s) \mu_n(s)
= \frac{1}{q^s} \sum_{n\ge q/t \atop p|n} \lambda_n(s) 
 \bigl(\mu_n(s)-\log q \bigr).
\end{equation}
As in \cite{CFAE}, we let $s=1+1/\log^2 N$ and split the sum on the left according to $n\le N$ and $n>N$. 
As $N\to \infty$, the contribution from $n\le N$ converges to (see \cite[Lemma 3]{CFAE})
$\sum_{ pq|n} \lambda_n \mu_n$,
while the contribution from $n>N$ converges to (see \cite[Lemma 4]{CFAE})
$
-c_{pq} (1-e^{-\gamma}), 
$
by Corollary \ref{corDqdv}.
Applying the same reasoning to the right-hand side of  \eqref{Derivative2}, we get 
$$
 \sum_{n\ge 1 \atop pq|n} \lambda_n \mu_n -c_{pq} (1-e^{-\gamma})= \frac{1}{q} \Bigl(\sum_{n\ge q/t \atop p|n} \lambda_n \bigl(\mu_n-\log q \bigr) - c_p(1-e^{-\gamma}) \Bigr).
 $$
We estimate $\mu_n$ with \eqref{muapprox} to obtain
\begin{equation}\label{cpqT}
 q c_{pq} (1-e^{-\gamma})= T\log q + c_p (1-e^{-\gamma}) 
  +O\left(T\exp\left(-\sqrt{\log \max(pt,q)}\right)\right),
 \end{equation}
 where
 $$
T:= \sum_{n\ge q/t \atop p|n} \lambda_n = q \sum_{n\ge 1 \atop pq|n} \lambda_n,
$$
by \eqref{lambdak}.

If $p\le q \le t$, then 
$$
T=  \sum_{n\ge 1 \atop p|n} \lambda_n = \frac{1}{p}  \sum_{n\ge p/t} \lambda_n =  \frac{1}{p}  \sum_{n\ge 1} \lambda_n
=\frac{1}{p},
$$
by \eqref{lambdak} and \eqref{lambda0}.
The estimate \eqref{cpqtasymp3} now follows from \eqref{cpqT} with $T=1/p$ and $c_p$ estimated by \eqref{cqtasymp2}.

In general, we have 
\begin{equation}\label{Test}
T= \frac{e^{-\gamma} c_p}{\log q} \left(1+O\left(\frac{1}{\log q} + \frac{\log^2 pt}{\log^2 q}\right)\right),
\quad (q \ge p \ge 2, \, t\ge 2). 
\end{equation}
If $q<pt$, this is implied by $T\ll c_p / \log q$, which follows from \eqref{eqcorDqdv2}  (with $q$ replaced by $p$). 
If $q\ge pt$, we estimate $\lambda_n$ with Mertens' formula and use Abel summation and the estimate \eqref{eqcorDqdv}.
The contribution from the first two error terms  in \eqref{eqcorDqdv} is clearly acceptable, while the term $O(1)$ contributes 
$$
\ll \int_{q/t}^\infty \frac{dy}{y^2 \log yt} \le \frac{t}{q \log q} \asymp \frac{c_p}{\log q} \cdot \frac{pt}{q \log t} 
\ll  \frac{c_p}{\log q} \cdot \frac{\log^2 pt}{\log^2 q}. 
$$

Now substitute \eqref{Test} into \eqref{cpqT} and 
estimate $c_p$ with \eqref{cqtasymp1} to get \eqref{cpqtasymp1},
and with \eqref{cqtasymp2} to get \eqref{cpqtasymp2}.
\end{proof}

\section{Proof of Theorem \ref{thmave}}\label{secpot1}

\begin{proof}[Proof of \eqref{eqave}.]
We have 
$$
\sum_{n\in \mathcal{D}(x)} \omega(n)
= \sum_{p\le x} D_p(x)
= x d(v) \sum_{p\le x} \eta_{p,t} \left(1+O\left(\frac{\log p}{\log x}\right)\right),
$$
by \eqref{eqDqdv}.
The contribution from the error term is $\ll x d(v) \asymp D(x)$, since $\eta_{p,t} \ll 1/p$ by \eqref{etaeq}. 
Now $\eta_{p,t} C \log t = c_p$ and $\eta_{1,t} C \log t = c_\theta$, by Corollary \ref{corDqdv}. 
With \eqref{c1tasymp}, \eqref{cqtasymp1} and \eqref{cqtasymp2}, we find that 
$$
\sum_{p\le x} \eta_{p,t} = \frac{1}{C \log t} \sum_{p\le x} c_{p} =  \frac{c_\theta}{C \log t}(E(x,t) +O(1))
= \eta_{1,t}  (E(x,t) +O(1)).
 $$  
 The result now follows from \eqref{eqDqdv} with $q=1$, that is  $D(x)= x d(v)\eta_{1,t} ( 1+O(1/\log x))$.

To see that \eqref{eqave} remains valid when $\omega$ is replaced by $\Omega$, note that
$$
\sum_{n\in \mathcal{D}(x)} (\Omega(n) -\omega(n))
= \sum_{n\in \mathcal{D}(x)} \sum_{k\ge 2} \sum_{p^k|n} 1
= \sum_{p\le x} \sum_{k\ge 2}  D_{p^k} (x) \ll \frac{x \log t}{\log xt} \ll D(x),
$$
by Lemma \ref{DqUB}. 
\end{proof}

\begin{proof}[Proof of \eqref{eqnorm}.]
We have 
$$ \sum_{n\in\mathcal{D}(x)} \omega(n)^2 
= \sum_{p, q\le x} D_{pq}(x) + O\Bigl(\sum_{p\le x} D_p(x)\Bigr),$$
where $p, q$ run over primes.
The last term is $\ll D(x) \log_2 x$, by \eqref{eqave}.
Thus,  \eqref{eqDqdv} yields
$$
 \sum_{n\in\mathcal{D}(x)} \omega(n)^2 = 
O( D(x) \log_2 x)+
 x d(v)\sum_{p, q \le x} \eta_{pq,t} \left(1+O\left(\frac{\log pq}{\log x}\right)\right).$$
Since $\eta_{pq,t}\ll 1/pq$, the contribution from the error term is $\ll x d(v)  \log_2 x \ll D(x) \log_2 x$. 
With \eqref{cpqtasymp1}, \eqref{cpqtasymp2} and \eqref{cpqtasymp3}, we find that 
\begin{equation*}
\begin{split}
\sum_{p,q \le x} \eta_{pq,t} = \frac{1}{C \log t} \sum_{p,q \le x} c_{pq} 
& =  \frac{c_\theta}{C \log t}\bigl(E(x,t)^2+O(\log_2 x)\bigr) \\ 
& = \eta_{1,t}  \bigl(E(x,t)^2 +O(\log_2 x)\bigr).
\end{split}
\end{equation*}
Combining this with \eqref{eqave} and \eqref{eqDqdv} (with $q=1$), we get
$$
 \sum_{n\in\mathcal{D}(x)} (\omega(n)-E(x,t))^2 
 \ll D(x) \log_2 x,
$$
which implies \eqref{eqnorm}.
This estimate remains valid if $\omega(n)$ is replaced by $\Omega(n)$, since
$$
\sum_{n\in\mathcal{D}(x)} (\Omega(n) - \omega(n))^2 
\le \sum_{p \neq q \le x} \sum_{k,j\ge 2} D_{p^k q^j}(x)
+ \sum_{p\le x} \sum_{k\ge 2} 2k D_{p^k}(x) \ll D(x),
$$
by Lemma \ref{DqUB}.
\end{proof}

\section*{Acknowledgments}
The author thanks Eric Saias and the anonymous referee for several very helpful suggestions.

\end{document}